\documentclass[12pt,reqno]{amsart}
\usepackage{amscd,amssymb,amsmath,multicol,tikz,float}
\begin{document}

\restylefloat{table}
\newtheorem{thm}[equation]{Theorem}
\numberwithin{equation}{section}
\newtheorem{cor}[equation]{Corollary}
\newtheorem{expl}[equation]{Example}
\newtheorem{rmk}[equation]{Remark}
\newtheorem{conv}[equation]{Convention}
\newtheorem{claim}[equation]{Claim}
\newtheorem{lem}[equation]{Lemma}
\newtheorem{sublem}[equation]{Sublemma}
\newtheorem{conj}[equation]{Conjecture}
\newtheorem{defin}[equation]{Definition}
\newtheorem{diag}[equation]{Diagram}
\newtheorem{prop}[equation]{Proposition}
\newtheorem{notation}[equation]{Notation}
\newtheorem{tab}[equation]{Table}
\newtheorem{fig}[equation]{Figure}
\newcounter{bean}
\renewcommand{\theequation}{\thesection.\arabic{equation}}

\raggedbottom \voffset=-.7truein \hoffset=0truein \vsize=8truein
\hsize=6truein \textheight=8truein \textwidth=6truein
\baselineskip=18truept

\def\mapright#1{\ \smash{\mathop{\longrightarrow}\limits^{#1}}\ }
\def\mapleft#1{\smash{\mathop{\longleftarrow}\limits^{#1}}}
\def\mapup#1{\Big\uparrow\rlap{$\vcenter {\hbox {$#1$}}$}}
\def\mapdown#1{\Big\downarrow\rlap{$\vcenter {\hbox {$\ssize{#1}$}}$}}
\def\mapne#1{\nearrow\rlap{$\vcenter {\hbox {$#1$}}$}}
\def\mapse#1{\searrow\rlap{$\vcenter {\hbox {$\ssize{#1}$}}$}}
\def\mapr#1{\smash{\mathop{\rightarrow}\limits^{#1}}}
\def\ss{\smallskip}
\def\s{\sigma}
\def\l{\lambda}
\def\vp{v_1^{-1}\pi}
\def\at{{\widetilde\alpha}}

\def\sm{\wedge}
\def\la{\langle}
\def\ra{\rangle}
\def\ev{\text{ev}}
\def\od{\text{od}}
\def\on{\operatorname}
\def\ol#1{\overline{#1}{}}
\def\spin{\on{Spin}}
\def\cat{\on{cat}}
\def\Lbar{\overline{\Lambda}}
\def\qed{\quad\rule{8pt}{8pt}\bigskip}
\def\ssize{\scriptstyle}
\def\a{\alpha}
\def\bz{{\Bbb Z}}
\def\Rhat{\hat{R}}
\def\im{\on{im}}
\def\ct{\widetilde{C}}
\def\ext{\on{Ext}}
\def\sq{\on{Sq}}
\def\eps{\epsilon}
\def\ar#1{\stackrel {#1}{\rightarrow}}
\def\br{{\bold R}}
\def\bC{{\bold C}}
\def\bA{{\bold A}}
\def\bB{{\bold B}}
\def\bD{{\bold D}}
\def\bC{{\bold C}}
\def\bh{{\bold H}}
\def\bQ{{\bold Q}}
\def\bP{{\bold P}}
\def\bx{{\bold x}}
\def\bo{{\bold{bo}}}
\def\dh{\widehat{d}}
\def\si{\sigma}
\def\Vbar{{\overline V}}
\def\dbar{{\overline d}}
\def\wbar{{\overline w}}
\def\Sum{\sum}
\def\tfrac{\textstyle\frac}

\def\tb{\textstyle\binom}
\def\Si{\Sigma}
\def\w{\wedge}
\def\equ{\begin{equation}}
\def\b{\beta}
\def\G{\Gamma}
\def\L{\Lambda}
\def\g{\gamma}
\def\d{\delta}
\def\k{\kappa}
\def\psit{\widetilde{\Psi}}
\def\tht{\widetilde{\Theta}}
\def\psiu{{\underline{\Psi}}}
\def\thu{{\underline{\Theta}}}
\def\aee{A_{\text{ee}}}
\def\aeo{A_{\text{eo}}}
\def\aoo{A_{\text{oo}}}
\def\aoe{A_{\text{oe}}}
\def\vbar{{\overline v}}
\def\endeq{\end{equation}}
\def\sn{S^{2n+1}}
\def\zp{\bold Z_p}
\def\cR{{\mathcal R}}
\def\P{{\mathcal P}}
\def\cQ{{\mathcal Q}}
\def\cj{{\cal J}}
\def\zt{{\bold Z}_2}
\def\bs{{\bold s}}
\def\bof{{\bold f}}
\def\bq{{\bold Q}}
\def\be{{\bold e}}
\def\Hom{\on{Hom}}
\def\ker{\on{ker}}
\def\kot{\widetilde{KO}}
\def\coker{\on{coker}}
\def\da{\downarrow}
\def\colim{\operatornamewithlimits{colim}}
\def\zphat{\bz_2^\wedge}
\def\io{\iota}
\def\om{\omega}
\def\Prod{\prod}
\def\e{{\cal E}}
\def\zlt{\Z_{(2)}}
\def\exp{\on{exp}}
\def\abar{{\overline a}}
\def\xbar{{\overline x}}
\def\ybar{{\overline y}}
\def\zbar{{\overline z}}
\def\mbar{{\overline m}}
\def\nbar{{\overline n}}
\def\sbar{{\overline s}}
\def\kbar{{\overline k}}
\def\bbar{{\overline b}}
\def\et{{\widetilde E}}
\def\ni{\noindent}
\def\tsum{\textstyle \sum}
\def\coef{\on{coef}}
\def\den{\on{den}}
\def\lcm{\on{l.c.m.}}
\def\Ext{\operatorname{Ext}}
\def\iso{\approx}
\def\lra{\longrightarrow}
\def\vi{v_1^{-1}}
\def\ot{\otimes}
\def\psibar{{\overline\psi}}
\def\thbar{{\overline\theta}}
\def\mhat{{\hat m}}
\def\exc{\on{exc}}
\def\ms{\medskip}
\def\ehat{{\hat e}}
\def\etao{{\eta_{\text{od}}}}
\def\etae{{\eta_{\text{ev}}}}
\def\dirlim{\operatornamewithlimits{dirlim}}
\def\gt{\widetilde{L}}
\def\lt{\widetilde{\lambda}}
\def\st{\widetilde{s}}
\def\ft{\widetilde{f}}
\def\sgd{\on{sgd}}
\def\lfl{\lfloor}
\def\rfl{\rfloor}
\def\ord{\on{ord}}
\def\gd{{\on{gd}}}
\def\rk{{{\on{rk}}_2}}
\def\nbar{{\overline{n}}}
\def\MC{\on{MC}}
\def\lg{{\on{lg}}}
\def\cH{\mathcal{H}}
\def\cS{\mathcal{S}}
\def\cP{\mathcal{P}}
\def\N{{\Bbb N}}
\def\Z{{\Bbb Z}}
\def\Q{{\Bbb Q}}
\def\R{{\Bbb R}}
\def\C{{\Bbb C}}
\def\Lb{\overline\Lambda}
\def\mo{\on{mod}}
\def\xt{\times}
\def\notimm{\not\subseteq}
\def\Remark{\noindent{\it  Remark}}
\def\kut{\widetilde{KU}}
\def\Eb{\overline E}
\def\*#1{\mathbf{#1}}
\def\0{$\*0$}
\def\1{$\*1$}
\def\22{$(\*2,\*2)$}
\def\33{$(\*3,\*3)$}
\def\ss{\smallskip}
\def\ssum{\sum\limits}
\def\dsum{\displaystyle\sum}
\def\la{\langle}
\def\ra{\rangle}
\def\on{\operatorname}
\def\proj{\on{proj}}
\def\od{\text{od}}
\def\ev{\text{ev}}
\def\o{\on{o}}
\def\U{\on{U}}
\def\lg{\on{lg}}
\def\a{\alpha}
\def\bz{{\Bbb Z}}
\def\eps{\varepsilon}
\def\bc{{\bold C}}
\def\bN{{\bold N}}
\def\bB{{\bold B}}
\def\bW{{\bold W}}
\def\nut{\widetilde{\nu}}
\def\tfrac{\textstyle\frac}
\def\b{\beta}
\def\G{\Gamma}
\def\g{\gamma}
\def\zt{{\Bbb Z}_2}
\def\zth{{\bold Z}_2^\wedge}
\def\bs{{\bold s}}
\def\bx{{\bold x}}
\def\bof{{\bold f}}
\def\bq{{\bold Q}}
\def\be{{\bold e}}
\def\lline{\rule{.6in}{.6pt}}
\def\xb{{\overline x}}
\def\xbar{{\overline x}}
\def\ybar{{\overline y}}
\def\zbar{{\overline z}}
\def\ebar{{\overline \be}}
\def\nbar{{\overline n}}
\def\ubar{{\overline u}}
\def\bbar{{\overline b}}
\def\et{{\widetilde e}}
\def\M{\mathcal{M}}
\def\lf{\lfloor}
\def\rf{\rfloor}
\def\ni{\noindent}
\def\ms{\medskip}
\def\Dhat{{\widehat D}}
\def\what{{\widehat w}}
\def\Yhat{{\widehat Y}}
\def\abar{{\overline{a}}}
\def\minp{\min\nolimits'}
\def\sb{{$\ssize\bullet$}}
\def\mul{\on{mul}}
\def\N{{\Bbb N}}
\def\Z{{\Bbb Z}}
\def\S{\Sigma}
\def\Q{{\Bbb Q}}
\def\R{{\Bbb R}}
\def\C{{\Bbb C}}
\def\Xb{\overline{X}}
\def\eb{\overline{e}}
\def\notint{\cancel\cap}
\def\cS{\mathcal S}
\def\cR{\mathcal R}
\def\el{\ell}
\def\TC{\on{TC}}
\def\GC{\on{GC}}
\def\wgt{\on{wgt}}
\def\Ht{\widetilde{H}}
\def\wbar{\overline w}
\def\dstyle{\displaystyle}
\def\Sq{\on{sq}}
\def\Om{\Omega}
\def\ds{\dstyle}
\def\tz{tikzpicture}
\def\zcl{\on{zcl}}
\def\bd{\bold{d}}
\def\cM{\mathcal{M}}
\def\io{\iota}
\def\Vb#1{{\overline{V_{#1}}}}
\def\Ebar{\overline{E}}
\def\lb{\,\begin{picture}(-1,1)(1,-1)\circle*{4.5}\end{picture}\ }
\def\lbb{\,\begin{picture}(-1,1)(1,-1)\circle*{8}\end{picture}\ }
\def\zp{\Z_p}
\def\bL{\bold{L}}
\def\st{1.732}

\title
{Geodesic complexity of a tetrahedron}
\author{Donald M. Davis}
\address{Department of Mathematics, Lehigh University\\Bethlehem, PA 18015, USA}
\email{dmd1@lehigh.edu}

\date{June 18, 2023}

\maketitle
\begin{abstract} We prove that the geodesic complexity of a regular tetrahedron exceeds its topological complexity by 1 or 2. The proof involves a careful analysis of minimal geodesics on the tetrahedron.\end{abstract}

\section{Introduction}\label{introsec}
In \cite{Far}, Farber introduced the concept of the {\it topological complexity}, $\TC(X)$, of a topological space $X$, which is the minimal number $k$ such that there is a partition
$$X\times X=E_0\sqcup E_1\sqcup\cdots\sqcup E_k$$ 
with each $E_i$ being locally compact and admitting a continuous function $\phi_i:E_i\to P(X)$
such that $\phi_i(x_0,x_1)$ is a path from $x_0$ to 
$x_1$.\footnote{Farber's original definition involved partitions into $k$ sets rather than $k+1$, but for technical reasons the definition here has become more common.}
Here $P(X)$ is the space of paths in $X$, and each $\phi_i$ is called a motion-planning rule. If $X$ is the space of configurations of one or more robots, this models the number of rules required to program the robots to move between any two configurations. 

In \cite{david}, Recio-Mitter suggested that if $X$ is a metric space, then we require that the paths $\phi_i(x_0,x_1)$ be minimal geodesics from $x_0$ to $x_1$, and defined the {\it geodesic complexity}, $\GC(X)$, to be the smallest number $k$ such that there is a partition 
$$X\times X=E_0\sqcup E_1\sqcup\cdots\sqcup E_k$$ 
with each $E_i$ being locally compact and admitting a continuous function $\phi_i:E_i\to P(X)$ such that $\phi_i(x_0,x_1)$ is a minimal geodesic from $x_0$ to $x_1$. Each function $\phi_i$ is called a {\it geodesic motion-planning rule} (GMPR).

One example discussed in \cite{david} was when $X$ is (the surface of) a cube. It is well-known that here $\TC(X)=\TC(S^2)=2$, and he showed that $\GC(X)\ge3$. 

In this paper, we let $X$ be a regular tetrahedron $T$, and prove
\begin{thm} $\GC(T)=3$ or $4$.\end{thm}
\ni Again, for comparison, $\TC(T)=\TC(S^2)=2$.

In Section \ref{ECLsec}, we introduce what we call the {\it expanded cut locus} in order to study the geodesics on $T$. In Section \ref{upsec}, we prove $\GC(T)\le4$, and in Section \ref{lowersec}, we prove $\GC(T)\ge3$. Despite considerable effort, we have been unable to establish the precise value of $\GC(T)$.

\section{Expanded cut locus}\label{ECLsec}

The {\it cut locus} of a point $P$ on a convex polyhedron is the set of points $Q$ such that there is more than one shortest path from $P$ to $Q$. For the regular tetrahedron $T$, this is conveniently sketched on a flat model of $T$. For $P\in T$, we define the {\it expanded cut locus} of $P$ to be the set of terminal points of equal shortest paths from $P$ to versions of cut-locus points $Q$ in a flat model of $T$, expanded so that the same face may appear more than once.

In Figure \ref{figC} we illustrate the expanded cut locus of a point $P$. The open segments $a\,U_0$ and $a\,U_-$ depict the same set of points in the tetrahedron, and the segments from $P$ to points on each at equal distance from $a$ depict equal shortest segments from $P$ to a point $Q$ in $T$. A similar situation holds for segments from $d$ to two $U$-points, from $c$ to two $L$-points, and from $b$ to two $L$-points. Also the small open segments $U_-\,L_-$ and $U_+\,L_+$ are part of the expanded cut locus of $P$, as they represent the same points in $T$, and segments from $P$ to points at equal height on the two lines are equal minimal geodesics.
The three $U$-points represent the same point in $T$; the paths from $P$ to them are equal shortest paths in $T$. Similarly for the three $L$-points. Thus the expanded cut locus of $P$ is the entire red polygon in Figure \ref{figC} minus the points $a$, $b$, $c$, and $d$.

The actual cut locus for this point $P$ is shown in Figure \ref{figB}, which is a flat version of part of $T$, but does not contain multiple versions of points.
 
\bigskip
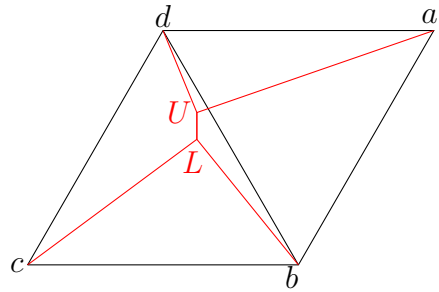
\begin{minipage}{6in}
\begin{fig}\label{figC}

{\bf An expanded cut locus.}

\begin{center}

\begin{\tz}[scale=1.8]
\draw (-3,0) -- (3,0) -- (1,2*\st) -- (-1,0) -- (0,-\st) -- (3,2*\st) -- (1,2*\st);
\draw (-3,0) -- (-2,\st) -- (2,\st);
\draw (-2,\st) -- (-1,0);
\draw [ultra thick] (1,0) -- (0,\st) -- (-1,0) -- (1,0);
\draw [dotted] (0,\st) -- (0,0);
\draw [dotted] (-1,0) -- (.5,.5*\st);
\draw [dotted] (1,0) -- (-.5,.5*\st);
\draw [color=red] (1.75,\st*1.35) -- (-1.75,\st*.65) -- (-1.75,\st*.5357) -- (-.25,-\st*.5357) -- (2.25,\st*.5357) -- (2.25,\st*.65) -- (1.75,\st*1.35);
\draw[fill](.25,\st*.5833) circle[radius=.03];
\node at (.16,\st*.5833) {$P$};
\draw[fill,color=red](-.25,-\st*.5357) circle[radius=.03];
\draw[fill,color=red](1.75,\st*1.35) circle[radius=.03];
\draw[fill,color=red](-1.75,\st*.65) circle[radius=.03];
\draw[fill,color=red](-1.75,\st*.5357) circle[radius=.03];
\draw[fill,color=red](2.25,\st*.5357) circle[radius=.03];
\draw[fill,color=red](2.25,\st*.65) circle[radius=.03];
\node at (-.03,\st*1.06) {$a$};
\node at (-1.05,-.09) {$b$};
\node at (1.05,-.08) {$c$};
\node at (3.06,0) {$b$};
\node at (-3.08,0) {$c$};
\node at (-2,\st*1.06) {$d$};
\node at (2.08,\st) {$d$};
\node at (1,2.06*\st) {$b$};
\node at (3.06,2*\st) {$c$};
\node at (0,-1.05*\st) {$d$};
\node at (0,.333*\st) {$C$};
\node at (.5,.5*\st) {$M$};
\node [color=red] at (-.25,-\st*.6) {$L_0$};
\node [color=red] at (-1.8,\st*.48) {$L_-$};
\node [color=red] at (-1.89,\st*.65) {$U_-$};
\node [color=red] at (1.75,\st*1.42) {$U_0$};
\node [color=red] at (2.27,\st*.46) {$L_+$};
\node [color=red] at (2.11,\st*.65) {$U_+$};
\end{\tz}
\end{center}
\end{fig}
\end{minipage}

\bigskip
\begin{minipage}{6in}
\begin{fig}\label{figB}

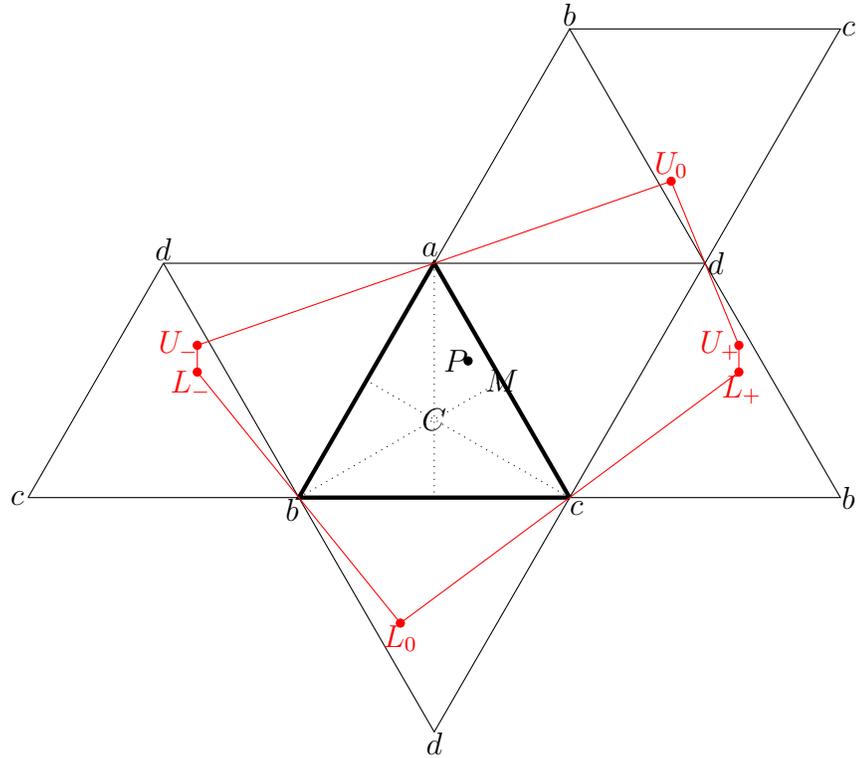
{\bf The corresponding cut locus.}

\begin{center}

\begin{\tz}[scale=1.8]
\draw (-3,0) -- (-1,0) -- (0,\st) -- (-2,\st) -- (-3,0);
\draw (-1,0) -- (-2,\st);
\node [color=red] at (-1.78,\st*.44) {$L$};
\node [color=red] at (-1.88,\st*.65) {$U$};
\draw [color=red] (-3,0) -- (-1.75,\st*.5357) -- (-1.75,\st*.65) -- (-2,\st);
\draw [color=red] (-1,0) -- (-1.75,\st*.5357) -- (-1.75,\st*.65) -- (0,\st);
\node at (-3.08,0) {$c$};
\node at (-2,\st*1.06) {$d$};
\node at (-.03,\st*1.06) {$a$};
\node at (-1.05,-.09) {$b$};
\end{\tz}
\end{center}
\end{fig}
\end{minipage}

\bigskip
The expanded cut locus of any point $P$ in the interior of  triangle $aCM$ in Figure \ref{figC}, where $C$ is the centroid and $M$ the midpoint of $ac$, has a form similar to the one depicted there. We make this precise in Theorem \ref{thm1}.

\begin{thm} Suppose that in Figure \ref{figC} the coordinates of $a$, $b$, and $c$ are, respectively, $(0,\sqrt3)$, $(-1,0)$, and $(1,0)$, and $P=(x,\a\sqrt3)$ with $0<x<\frac12$ and $\frac13+\frac13x<\a<1-x$. Then the expanded cut locus of $P$ is as depicted in Figure \ref{figC} and described above with
\begin{eqnarray}U_{\pm}&=&(\pm2+x,\sqrt3(1-\frac{x(2-x)}{3(1-\a)}))\nonumber\\
U_0&=&(2-x,\sqrt3(1+\frac{x(2-x)}{3(1-\a)}))\nonumber\\
L_\pm&=&(\pm2+x,\sqrt3\ \frac{1-x^2}{3\a})\label{UL}\\
L_0&=&(-x,\sqrt3\ \frac{x^2-1}{3\a})\nonumber\end{eqnarray}\label{thm1}
\end{thm}

\begin{proof} Since
$$\langle x,\sqrt3(\a-1)\rangle\cdot\langle 2-x,\sqrt3\ \frac{x(2-x)}{3(1-\a)}\rangle=0,$$
$\overrightarrow{aP}\perp \overrightarrow{aU_0}$. Similarly the red lines through $b$, $c$, and $d$ are perpendicular to the segments from $P$ to those points. Another easy verification is that $\frac12(U_0+U_-)=a$, and similarly for $b$, $c$, and $d$. 
So $Pa$ is the perpendicular bisector of $U_0U_-$.
That $\frac12(U_0+U_+)=d$ shows that $U_0$ and $U_+$ lie in the same relative position in triangle $bcd$. 
One readily sees that the region inside the red polygon in Figure \ref{figC} exactly covers the four triangles that comprise the tetrahedron.\end{proof}

This slick verification hides the way in which the formulas (\ref{UL}) were obtained. We initially used the method of star unfolding and Voronoi diagrams developed in \cite{star97}, and applied to the cube in \cite{DM}, using perpendicular bisectors.

The triangle $abc$ in Figure \ref{figC} is divided into six congruent subtriangles. The formulas (\ref{UL}) only apply to points $P$ in the interior of the upper right subtriangle $aCM$, but the expanded cut locus of points in the other five subtriangles can be obtained by obvious rotations and reflections.
We now consider the form of the expanded cut locus for points on the boundary of triangle $aCM$.

As $P$ approaches the edge $aM$, $L_\pm$ approaches $U_\pm$. When $P$ is on the edge, they coincide, and the two multiplicity-3 points in the cut locus become a single multiplicity-4 point, which we will later call $B$, for ``both.'' In Figure \ref{fig2}, we depict the two extreme cases, $P=a$ and $P=M$. The continuum between them should be clear. We label the left one $P\approx a$, because when $P=a$, the line passing through $a$ is not part of the expanded cut locus, since the line connecting $P$ with points on the lines at equal distance from $a$ in each direction are actually the same line in $T$. But for points $P$ arbitrarily close to $a$, the lines from $P$ to points on the line are not the same line in $T$.

\bigskip
\begin{minipage}{6in}
\begin{fig}\label{fig2}

{\bf $P$ on an edge.}

\begin{center}

\begin{\tz}[scale=.9]
\draw (-3,0) -- (3,0) -- (1,2*\st) -- (-1,0) -- (0,-\st) -- (3,2*\st) -- (1,2*\st);
\draw (-3,0) -- (-2,\st) -- (2,\st);
\draw (-2,\st) -- (-1,0);
\draw [ultra thick] (1,0) -- (0,\st) -- (-1,0) -- (1,0);
\draw (4,0) -- (10,0) -- (8,2*\st) -- (6,0) -- (7,-\st) -- (10,2*\st) -- (8,2*\st);
\draw (4,0) -- (5,\st) -- (9,\st);
\draw (5,\st) -- (6,0);
\draw [ultra thick] (8,0) -- (7,\st) -- (6,0) -- (8,0);
\draw [color=red] (0,-\st*.333) -- (2,\st*.333) -- (2,\st*1.667) -- (-2,.333*\st) -- (0,-\st*.333);
\draw [color=red] (6.5,-.5*\st) -- (9.5,.5*\st) -- (8.5,1.5*\st) -- (5.5,.5*\st) -- (6.5,-.5*\st);
\node at (-1,1.5*\st) {$P\approx a$};
\node at (6,1.5*\st) {$P=M$};
\node at (-.09,1.09*\st) {$a$};
\node at (7.5,.5*\st) {$M$};
\end{\tz}
\end{center}
\end{fig}
\end{minipage}

\bigskip

As $P$ approaches the line $x=0$, $U_+$ and $U_0$ approach $d_+$ (the version of $d$ on the positive side in Figure \ref{figC}), and $U_-$ approaches $d_-$. The diagram when $x=0$ is in Figure \ref{fig0}. 

\bigskip
\begin{minipage}{6in}
\begin{fig}\label{fig0}

{\bf $P$ on the line $x=0$.}

\begin{center}

\begin{\tz}[scale=1.5]
\draw [ultra thick] (-1,0) -- (1,0) -- (0,\st) -- (-1,0);
\draw (0,.333*\st) -- (0,\st);
\draw (-1,0) -- (-2,\st) -- (-3,0) -- (-1,0) -- (0,-\st) -- (1,0) -- (3,0) -- (2,\st) -- (1,0);
\draw [color=red] (-2,\st) -- (-2,.5*\st) -- (0,-.5*\st) -- (2,.5*\st) -- (2,\st) -- (-2,\st);
\node at (0,1.05*\st) {$a$};
\node at (2,1.07*\st) {$d=U$};
\node at (-2,1.07*\st) {$d=U$};
\node at (-3,-.05*\st) {$c$};
\node at (3,-.06*\st) {$b$};
\node at (0,-1.05*\st) {$d$};
\node at (-1.05, -.06*\st) {$b$};
\node at (1.05,-.05*\st) {$c$};
\node at (-2.06, .5*\st) {$L$};
\node at (2.06,.5*\st) {$L$};
\node at (0,-.4*\st) {$L$};
\node at (.08,.33*\st){$C$};
\end{\tz}
\end{center}
\end{fig}
\end{minipage}
\bigskip

 As $P$ moves from $a$ to $C$ along the line $x=0$, the point $L$ in Figure \ref{fig0} moves from the centroid of $bcd$ to $d$. The limiting case $P=a$ has already been discussed. However, if the $L$ in Figure \ref{fig0} is moved to the centroid of $bcd$, we obtain a picture which looks quite different from the left side of Figure \ref{fig2}, which also depicts the case $P=a$. Even accounting for the fact that when $P=a$, the line emanating from $a$ is not part of the expanded cut locus, the diagrams still  differ in that one has a vertical line on the left side, whereas the other has a vertical line in the upper right. The explanation is that paths from $a$ to corresponding points on those lines are exactly the same path on $T$.
 
In Figure \ref{cent} we show the expanded cut locus when $P$ is at the centroid $C$ of $abc$, which is the case $L=d$ in Figure \ref{fig0}. 

\bigskip
\begin{minipage}{6in}
\begin{fig}\label{cent}

{\bf $P$ at the centroid.}

\begin{center}

\begin{\tz}[scale=.8]
\draw [ultra thick] (-1,0) -- (1,0) -- (0,\st) -- (-1,0);
\draw [color=red] (-2,\st) -- (0,-\st) -- (2,\st) -- (-2,\st);
\draw[fill](0,\st*.333) circle[radius=.03];
\node at (0,-1.06*\st) {$U=L$};
\node at (2.2,\st*1.07) {$U=L$};
\node at (-2.2,\st*1.07) {$U=L$};
\end{\tz}
\end{center}
\end{fig}
\end{minipage}
\bigskip

Finally, if $P$ is on the segment $CM$, $U_\pm=L_\pm(=B)$, and they lie on edge $bd$. This is depicted in Figure \ref{CM}. As $P$ moves from $C$ to $M$, $B$  moves from $d$ to the midpoint of $bd$.

\bigskip
\begin{minipage}{6in}
\begin{fig}\label{CM}

{\bf $P$ on the segment $CM$.}

\begin{center}

\begin{\tz}[scale=1.5]
\draw [ultra thick] (-1,0) -- (1,0) -- (0,\st) -- (-1,0);
\draw [color=red] (-.25,-.75*\st) -- (2.25,.75*\st) -- (1.75,1.25*\st) -- (-1.75,.75*\st) -- (-.25,-.75*\st);
\draw (-.25,-.75*\st) -- (0,-\st) -- (1,0) -- (3,0) -- (2.25,.75*\st);
\draw (0,.333*\st) -- (.5,.5*\st);
\draw (1.75,1.25*\st) -- (1,2*\st) -- (0,\st) -- (-2,\st) -- (-1.75,.75*\st);
\draw (0,\st) -- (2,\st);
\node at (0,-1.06*\st) {$d$};
\node at (3.08,0) {$b$};
\node at (1.08,-.08) {$c$};
\node at (-1.08,0) {$b$};
\node at (-2.05,1.05*\st) {$d$};
\node at (-.05,1.05*\st) {$a$};
\node at (2.09,\st) {$d$};
\node at (1,2.06*\st) {$b$};
\node at (-.37,-.75*\st) {$B$};
\node at (1.86,1.28*\st) {$B$};
\node at (2.33,.75*\st) {$B$};
\node at (-1.86,.75*\st) {$B$};
\node at (-.08,.33*\st) {$C$};
\node at (.57,.5*\st) {$M$};
\end{\tz}
\end{center}
\end{fig}
\end{minipage}

\section{Upper bound}\label{upsec}
\begin{thm} There is a decomposition\label{upthm}
$$T\times T=E_1\sqcup E_2\sqcup E_3\sqcup E_4\sqcup E_5$$ with a GMPR $\phi_i$ on $E_i$.
\end{thm}
\begin{proof} Let $G_P$ denote the polygon associated to the point $P$ sketched in red in any of the figures of Section \ref{ECLsec}.
More precisely, one must, of course, use the formulas (\ref{UL}) to determine the vertices of the polygon, and if $P$ is reflected across the line $x=0$ in Figure \ref{figC}, then one must modify the formulas to give the reflection of the polygon. If $P$ is at a vertex, there are two choices for $G_P$, either as in Figure \ref{fig2} or \ref{fig0}. It doesn't matter, but let's  choose \ref{fig0}.

The set $E_1$ is the complement of the total cut locus of $T$. It consists of pairs $(P,Q)$ such that $Q$ is interior to the polygon $G_P$, together with those for which $Q$ is a vertex of $T$, except for cases such as $(P,d)$ in Figure \ref{fig0}. (The only cases when a vertex $V$ is in the cut locus of a point $P$ is when $P$ lies on a segment connecting the centroid of the face opposite $V$ with one of the other vertices, including the centroid, but not the vertices.) Here $\phi_1(P,Q)$ is the straight line from $P$ to $Q$ in our expanded cut locus diagram.

The set $E_2$ consists of pairs $(P,Q)$ where $P$ is not a vertex and $Q$ lies  in the interior of a cut-locus segment from a vertex $V$ to a $U$ or $L$ point, excluding cases in which $P$ lies on a segment from a vertex of face $abc$ to the centroid $C$ of $abc$, and $V=d$. We choose $\phi_2(P,Q)$ to be the path from $P$ to the appropriate point on the right side of the vector from $P$ to $V$. For example, in Figures \ref{figC} and \ref{figB}, $E_2$ contains $(P,Q)$ for all $Q$ in the open segments $aU$, $bL$, $cL$, and $dU$ in \ref{figB}, and in \ref{figC} we choose the segments connecting $P$ with points on $aU_0$, $bL_-$, $cL_0$, and $dU_+$. To maintain continuity of $\phi_2$, we had to exclude points $(P,Q)$ with $P$ on the segment $aC$ and $Q$ on $dU$ because shortest paths from the point $P$ in Figure \ref{figC} to $dU$ must pass through side $ac$, whereas for points $P$ on the left side of $aC$ the diagram is reflected and the shortest paths from $P$ to $dU$ will pass through side $ab$.

This requires some care because, for example, if $P$ is in face $abc$, the cut-locus line out from vertex $d$ plays a different role than the others. Because we have excluded points with $P$ on segments from a vertex to a centroid, we can consider the domain of points $P$ for which $Q$ is on a cut-locus line from vertex $d$ as three topologically disjoint sets $aCbd$, $adcC$, and $bCcd$, as pictured in Figure \ref{3d}.

\bigskip
\begin{minipage}{6in}
\begin{fig}\label{3d}

{\bf $P$-domains for lines through $d$.}

\begin{center}

\begin{\tz}[scale=.9]
\draw (-2,\st*2) -- (2,\st*2) -- (0,0) -- (-2,\st*2);
\draw (1,\st) -- (0,2*\st) -- (-1,\st) -- (1,\st);
\draw [dashed] (1,\st) -- (0,1.333*\st) -- (0,2*\st);
\draw [dashed] (0,1.333*\st) -- (-1,\st);
\node at (0,1.333*\st) {$C$};
\node at (0,-.09*\st) {$d$};
\node at (2.09,2*\st) {$d$};
\node at (-2.09,2*\st) {$d$};
\node at (0,2.09*\st) {$a$};
\node at (1.1,\st) {$c$};
\node at (-1.1,\st) {$b$};
\end{\tz}
\end{center}
\end{fig}
\end{minipage}
\bigskip

The continuity of $\phi_2$ on each of these domains should be fairly clear, but because of the different roles played by points in face $abc$ and the other points, Figure \ref{bigfig} should make it clearer. What is pictured here is a breakdown of the region $aCcd$ in Figure \ref{3d} into subregions together with, for each subregion, the  endpoints of the cut-locus segments out of vertex $d$ corresponding to points $P$ in the subregion. For example, output region 2 is points $U_+$
in Figure \ref{figC} corresponding to points in input region 2, and output region 6 is points $U_0$ in a rotated version of Figure \ref{figC} corresponding to points in input region 6. The entire segment between input regions 5 and 6 maps to output point $b$. The dashed boundary of output regions 5 and 6 are not in the image. We call the points $Q_{\text{max}}$ in Figure \ref{bigfig} because they are the $Q$ farthest from $d$ for a point $P$.

\bigskip
\begin{minipage}{6in}
\begin{fig}\label{bigfig}

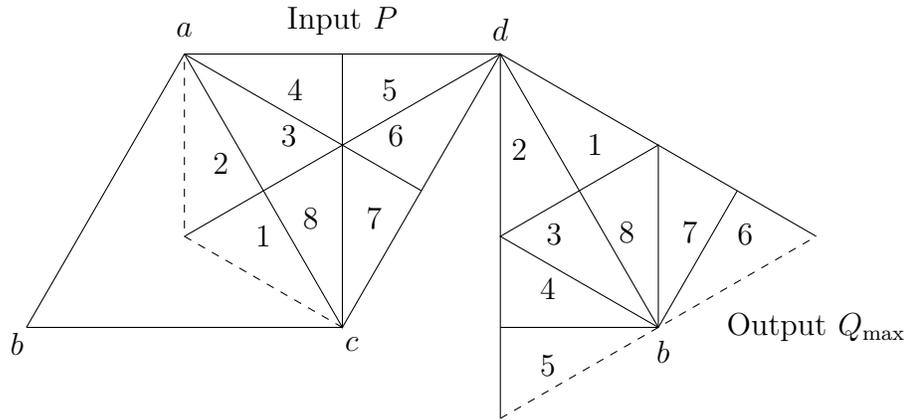
{\bf Largest $Q$ for varying $P$.}

\begin{center}

\begin{\tz}[scale=2.1]
\draw (-1,0) -- (1,0) -- (0,\st) -- (-1,0);
\draw (0,\st) -- (2,\st) -- (1,0) -- (1,\st);
\draw [dashed] (1,0) -- (0,.333*\st) -- (0,\st);
\draw (0,.333*\st) -- (2,\st) -- (2,-.333*\st);
\draw (0,\st) -- (1.5,.5*\st);
\draw [dashed] (2,-.333*\st) -- (4,.333*\st);
\draw (3,0) -- (3,.667*\st) -- (2,.333*\st) -- (3,0);
\draw (2,0) -- (3,0) -- (3.5,.5*\st);
\draw (4,.333*\st) -- (2,\st) -- (3,0);
\node at (.5,.333*\st) {$1$};
\node at (.8,.4*\st) {$8$};
\node at (1.2,.4*\st) {$7$};
\node at (.66,.7*\st) {$3$};
\node at (1.34,.7*\st) {$6$};
\node at (1.3,.87*\st) {$5$};
\node at (.7,.87*\st) {$4$};
\node at (.23,.6*\st) {$2$};
\node at (2.12, .65*\st) {$2$};
\node at (2.6,.67*\st) {$1$};
\node at (2.34,.34*\st) {$3$};
\node at (2.8,.35*\st) {$8$};
\node at (2.3, .14*\st) {$4$};
\node at (2.3,-.14*\st) {$5$};
\node at (3.2,.35*\st) {$7$};
\node at (3.55, .34*\st) {$6$};
\node at (3.03,-.09*\st) {$b$};
\node at (2,1.09*\st) {$d$};
\node at (1.06,-.06*\st) {$c$};
\node at (-1.06,-.06*\st) {$b$};
\node at (0,1.09*\st) {$a$};
\node at (1,1.12*\st) {{\rm Input} $P$};
\node at (4,0) {{\rm Output} $Q_{\max}$};
\end{\tz}
\end{center}
\end{fig}
\end{minipage}
\bigskip



The set $E_3$ consists of points $(P,Q)$ of two types. Type (1) has $P$ in sets $\mathcal{I}$ defined as the interior 
of the set of points in a face which are closer to one vertex than to the others. For example, in Figure \ref{figC}, one such region would be the interior of the quadrilateral in the upper third of triangle $abc$. The points $Q$ associated to $P$ are the closed interval $UL$. Type (2) has $P$ all points on segments connecting a vertex $V$ of a face $abc$ with its centroid $C$, including $V$ but not $C$, and $Q$ in the closed segment  connecting the other vertex $d$ with the point $L$ associated with $P$ as in Figure \ref{fig0}. Note that this can be considered as a $UL$ segment, too.

For $P\in\mathcal{I}$ and $Q$ in the closed interval $UL$, we can choose $\phi_3(P,Q)$ to be the appropriate point in $U_+L_+$, using rotations of Figure \ref{figC}. Then in Figure \ref{fig0}, we would choose as $\phi_3(P,Q)$ the path that goes to the right from a point $P$ on $aC$ to the appropriate $Q$ on $dL$.

The rest is easy. Let $E_4$ consist of pairs $(P,Q)$ such that $P$  is a vertex and $Q$ the centroid of the opposite face, or $P$ is a centroid and $Q$ the opposite vertex. Since this is a discrete set, $\phi_4$ can be chosen arbitrarily.

Let $E_5$ be the set of $(P,Q)$ such that $P$ lies in one of six topologically disjoint sets, each of which is the union of lines from the midpoint $M$ of an edge of $T$ to the adjacent vertices and centroids, including $M$ but not the vertices or centroids. See Figure \ref{E5}. A unique point $Q=B$ is associated to each point $P$. Recall that when $U=L$, we call it $B$. These are points of multiplicity 4, as in Figures \ref{fig2} and \ref{CM}. As long as one chooses $\phi_5$ continuously on a set such as Figure \ref{E5}, it can be chosen arbitrarily.

\bigskip
\begin{minipage}{6in}
\begin{fig}\label{E5}

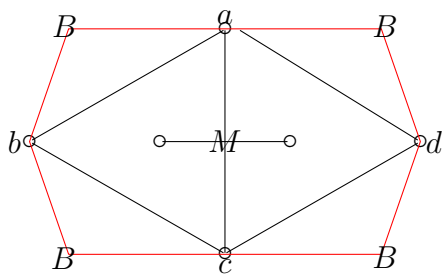
{\bf Typical set for $E_5$.}

\begin{center}

\begin{\tz}[scale=.5]
\draw (.04,-2.96) -- (2.96*\st,-.04);
\draw (2.96*\st,.04) -- (.4,2.96);
\node at (3*\st,0) {$\circ$};
\draw (-.04,2.96) -- (-2.96*\st,.04);
\draw (-2.96*\st,-.04) -- (-.04,-2.96);
\node at (0,3) {$\circ$};
\node at (-3*\st,0) {$\circ$};
\node at (0,-3) {$\circ$};
\draw (0,-2.96) -- (0,2.96);
\draw (-.96*\st,0) -- (.96*\st,0);
\node at (-\st,0) {$\circ$};
\node at (\st,0) {$\circ$};
\draw [color=red] (-2.4*\st,-3) -- (2.4*\st,-3) -- (3*\st,0) -- (2.4*\st,3) -- (-2.4*\st,3) -- (-3*\st,0) -- (-2.4*\st,-3);
\node at (0,3.32) {$a$};
\node at (0,-3.32) {$c$};
\node at (3.21*\st,0) {$d$};
\node at (-3.23*\st,0) {$b$};
\node at (2.46*\st,3.07) {$B$};
\node at (2.46*\st,-3.07){$B$};
\node at (-2.46*\st,3.07){$B$};
\node at (-2.48*\st,-3.09){$B$};
\node at (0,0) {$M$};
\end{\tz}
\end{center}
\end{fig}
\end{minipage}
\bigskip

    \end{proof}
\section{Lower bound}\label{lowersec}

\begin{thm} The space $T\times T$ cannot be partitioned as $E_1\sqcup E_2\sqcup E_3$ with a GMPR on each $E_1$.
\end{thm}
\begin{proof}
 Let $M$ be the midpoint of $ac$ in Figure \ref{figC}, and $P'$ a point on the segment connecting $M$ and $P$ in that figure. The expanded cut locus for $P'$ is as in the figure, and as $P'$ approaches $M$, $L_{\pm}$ approaches $U_{\pm}$, and they and $U_0$ and $L_0$ approach the midpoint of $bd$. We call this point $B$.

Suppose $(M,B)\in E_1$, and $\phi_1(M,B)$ is the path which goes down (toward the limit of $L_0$ in Figure \ref{figC}). (Going up is handled similarly, reversing the roles of $U$ and $L$. We will consider later how to handle it when $\phi_1(M,B)$ goes left or right.) We cannot have a sequence of $P'$ as in the figure with $P'\to M$ and $(P',U_{P'})\in E_1$ because that would imply $\phi_1(P',U_{P'})\to \phi_1(M,B)$, which is impossible since $\phi(P',U_{P'})$ must go either left, right, or up. There is a sequence of such $P_n'$ all in the same $E_i$, which we call $E_2$, and, restricting more, all $\phi_2(P_n',U_{P_n'})$ going in the same direction, which we will suppose is left. We will consider later the minor modifications required if $\phi_2(P_n',U_{P_n'})$ goes right or up.

For each such $P_n'$, there is an interval of $Q$'s in the cut locus of $P_n'$ abutting $U_{P_n'}$ along the segment from $d$ to $U_{P_n'}$. (It is close to $U_0$ and $U_+$ in Figure \ref{figC}.) There cannot be infinitely many of these with $(P_n',Q)\in E_2$ since $\phi(P_n',Q)$ must go right or up, but $\phi_2(P_n',U_{P_n'})$ goes left. If there were, for infinitely many $n$, a sequence $Q_{n,m}$ approaching $U_{P_n'}$ with $(P_n,Q_{n,m})\in E_1$, then the sequence $(P_n,Q_{n,n})$ would approach $(M,B)$, but $\phi_1(P_n,Q_{n,n})$ cannot approach $\phi_1(M,B)$, since the possible directions differ. Thus there exist sequences $Q_{n,m}\to U_{P_n'}$ with $(P_n',Q_{n,m})$ in a new set $E_3$, and we may assume that $\phi_3(P_n',Q_{n,m})$ all have the same direction, which we may assume to be ``up,'' i.e., toward the vicinity of $U_0$.

For each $(n,m)$, there exists a sequence $Q_{n,m,\ell}\to Q_{n,m}$ such that the unique minimal geodesic from $P_n'$ to $Q_{n,m,\ell}$ goes to the right, i.e., in the vicinity of $U_+$.
For each $(n,m)$, there cannot be infinitely many $\ell$ with $(P_n',Q_{n,m,\ell})\in E_3$, since $\phi_3(P_n',Q_{n,m})$ and $\phi(P_n',Q_{n,m,\ell})$ have different directions. We restrict now to, for each $(n,m)$, an infinite sequence of $\ell$ such that $(P_n',Q_{n,m,\ell})\not\in E_3$. Taking a diagonal limit on $m$ and $\ell$, $(P_n',Q_{n,m,\ell})\to (P_n',U_{P_n'})$; since $\phi_2(P_n',U_{P_n'})$ and $\phi(P_n',Q_{n,m,\ell})$ have opposite directions, $(P_n',Q_{n,m,\ell})\not\in E_2$ for an infinite sequence of $m$'s and all $\ell\ge L_m$ for an increasing sequence of integers $L_m$. Now taking a diagonal limit over  $n$, $m$, and $\ell$, we approach $(M,B)$. Since the directions of $\phi_1(M,B)$ and $\phi(P_n',Q_{n,m,\ell})$ differ, there must be an infinite sequence of $(P_n',Q_{n,m,\ell})$ not in $E_1$. So it requires a fourth set $E_4$.

Now we discuss the minor changes for other cases to which we alluded above. If $\phi_2(P_n',U_{P_n'})$ went right, instead of left, then the $Q$'s will be chosen on the segment from vertex $a$ to $U_{P_n'}$, close to $U_0$ and $U_-$, and the rest of the argument proceeds similarly.

If instead of going down or up, $\phi_1(M,B)$ goes left, then we consider $P'$ on a little segment going sharply down and left from $M$. The expanded cut locus will be similar to that in Figure \ref{figC}, but with $U_+L_+$ and $U_0$ interchanged (and moved slightly to the other side of line $bdb$), and similarly for $U_-L_-$ and $L_0$. These $P'$ have $\phi(P',U_{P'})$ going up, down, or right, and an argument like the one above works.
\end{proof}
\def\line{\rule{.6in}{.6pt}}

\end{document}